\documentclass[12pt, reqno]{amsproc}
\usepackage{graphicx, enumerate}
\usepackage[margin=1in]{geometry}
\usepackage{amsmath}
\usepackage{amsfonts}
\usepackage{amssymb}
\usepackage{amsfonts,amsmath,amssymb}
\usepackage{amsthm}
\usepackage[symbol]{footmisc}

\newtheorem{lemma}{Lemma}
\newtheorem{theorem}{Theorem}

\newtheorem{example}{Example}
\title[Solutions of Non-Homogenous Linear Differential Equations]{Solutions of Non-Homogenous Linear Differential Equations}
\author[Naveen Mehra and S. K. Chanyal]{Naveen Mehra and S. K. Chanyal}
\address{Naveen Mehra, Department of Mathematics, Kumaun University, D.S.B. Campus, Nainital-263001, Uttarakhand, India}
\email{naveenmehra00@gmail.com}
\address{S. K. Chanyal, Department of Mathematics, Kumaun Univesity, D.S.B. Campus, Nainital-263001, Uttarakhand, India.}
\email{skchanyal.math@gmail.com}
\thanks{The research work of the first author is supported by a research fellowship from the Council of Scientific and Industrial Research(CSIR), New Delhi.} 
\begin{document}
	\maketitle
		\begin{abstract}
This article is devoted to the study of solutions of non-homogenous linear  differential equations having entire coefficients. We get all non-trivial solutions of infinite order of equation $f^{(n)}+a_{n-1}(z)f^{(n-1)}+\ldots +a_1f'+a_0f=H(z)$, by restricting the conditions on coefficients.
	\end{abstract}
	
	\section{Introduction}
The results of this article depends heavily on the concepts of value distrbution theory. We discuss some basic notations here, for detailed theory, reader can see \cite{laine1993nevanlinna}. Let $(a_n)_{n\in\mathbb{N}}$ be a sequence of non-zero roots of function $f(z)$ with $\lim_{n\to\infty}a_n=\infty$, then exponent of convergence is defined as   
\begin{center}
	$\displaystyle{\lambda(f)= \inf\lbrace\alpha>0\displaystyle{|}\sum\limits_{n=1}^{\infty}|a_n|^{-\alpha}<\infty\rbrace}$.
\end{center}
The order of growth for an entire function $f(z)$ is given by
\begin{center}
	$\displaystyle{\rho(f)=\overline{\lim_{r \to \infty}}\,\displaystyle{\frac{\log^+\log^+{M(r,f)}}{\log{r}}}},$
\end{center}	
where $M(r,f)$ is the maximum of the function $f(z)$ on the disc of radius $|z|=r.$ The $n^{th}$ order homogenous linear differential equation for function $f$ with an entire coefficients $a_0(z), a_1(z), \ldots, a_{n-1}(z)$ is
\begin{equation}\label{norder}
	f^{(n)}+a_{n-1}(z)f^{(n-1)}+\ldots +a_1(z)f'+a_0(z)f=0
\end{equation}
and the associated non-homogenous linear differential equation is
\begin{equation}\label{nonhomonorder}
	f^{(n)}+a_{n-1}(z)f^{(n-1)}+\ldots +a_1f'+a_0f=H(z),
\end{equation}
where $H(z)$ is an entire function. If all the coefficients of equation \eqref{norder} are polynomials, then all solutions are of finite order. The solutions of associated non-homogenous equation \eqref{nonhomonorder} are of finite order if $H(z)$ is of finite order(see \cite[Lemma 2]{shian}). If $a_p$ be the last coefficient which is transcendental entire function, then atmost $p$ linearly independent solutions of equation \eqref{norder} are of finite order. Thus, if atleast one of the coefficient is transcendental entire function, then atmost all solutions of equation \eqref{norder} and \eqref{nonhomonorder} are of infinite order. Let $\rho$ be the minimal order of solutions of equation \eqref{norder}, then there may exist atmost one solution of order $<\rho$ of equation \eqref{nonhomonorder}. Thus, if all non-trivial solution of equation \eqref{norder} are of infinite order, there may exist finite order solution of equation \eqref{nonhomonorder}, we have illustrated this by following example. 
\begin{example}
	The equation $$f''+z f' + e^z f=e^{-z}(1-z) +1$$ has a finite order solution that is $f(z)=e^{-z},$ whereas the associated homogenous equation has all non-trivial solutions of infinite order.
\end{example}

\begin{example}
	Let $b(z)$ be a finite order entire function and has multiply connected Fatou component.
Then, the equation $$f''-e^zf'+b(z)f=0,$$ has all non-trivial solution of infinite order(see \cite[Theorem B]{mehra}). But the associated non-homogenous equation $$f''-e^zf'+b(z)f=e^{-z}(1+b(z))+1$$ has finite order solution $f(z)=e^{-z}.$
\end{example}
The main aim of this article is to find such conditions on coefficients of non-homogenous equation \eqref{nonhomonorder} to ensure that all non-trivial solutions are of infinite order.
\section{Lemmas}
The following lemma is given by Gundersen \cite{gundersen}, plays very crucial role in proving our results. It provides an estimate of logarithmic derivative of finite order transcendental meromorphic function.
\begin{lemma}\cite{gundersen}\label{lgundersen} Let $f$ be a transcendental meromorphic function with finite order and $(k, j)$ be a finite pair of integers that satisfies $k > j \geq 0.$ Let $\epsilon>0$ be a given constant. Then, the following three statements hold:
	\begin{enumerate}
		\item[(i)] there exists a set $E_1 \subset [0, 2\pi)$ that has linear measure zero, such that: if $\psi_0 \in [0, 2\pi) -E_1$, there is a constant $R_0 = R_0(\psi_0) > 0$ such that for all z satisfying $\arg z = \psi_0$ and $|z| \geq R_0$ and for all $(k, j) \in \Gamma$,
		\begin{equation}\label{f"byf}
			\left|\frac{f^{(k)}(z)}{f^{(j)}(z)}\right| \leq |z|^{(k-j)(\rho-1+\epsilon)}.
		\end{equation}
		\item[(ii)] there exists a set $E_2 \subset (1,\infty)$ that has finite logarithmic measure, such that for all z with $|z| \notin E_2 \cup[0, 1]$ and for all $(k, j) \in \Gamma$,
		the inequality \eqref{f"byf} holds.
		\item[(iii)] there exists a set $E_3 \subset [0,\infty)$ that has finite linear measure, such that for all z with $|z| \notin E_3$ and for all $(k, j) \in \Gamma$,
		\begin{equation}\label{f"byf-2}
			\left|\frac{f^{(k)}(z)}{f^{(j)}(z)}\right| \leq |z|^{(k-j)(\rho+\epsilon)}.
		\end{equation}
	\end{enumerate}
\end{lemma}

If an entire function $A(z)$ is of finite order of growth and satisfies $\lambda <\rho,$ then it can be written in the form $A(z)=h(z)e^{P(z)},$ where $h(z)$ is an entire function, $P(z)$ is a polynomial of degree $n$ and $\rho(h)<\deg (P).$ Let $P(z)=a_nz^n+\ldots +a_0$ and $z=re^{\iota\theta}$ then the notation $\delta$ is given as $\delta(P, \theta)=\Re(a_ne^{\iota n\theta}).$
The next lemma gives an estimates of an entire function of integral order. We have used this Lemma to prove Theorem \ref{mainthm5} and \ref{mainthm3}.

\begin{lemma}\cite{banklainelangley}\label{lcritical}
	Assume $g(z) = h(z)e^{P(z)}$ is an entire function for $z=re^{\iota\theta}$ satisfying $\lambda(g) < \rho(g) = n$, where $P(z)$ is a polynomial of degree $n$. Then, for
	every $\epsilon > 0$, there exists $E \subset [0, 2\pi)$ of linear measure zero satisfying
	\begin{enumerate}[(i)]
		\item for $\theta \in [0, 2\pi)\setminus E$ with $\delta(P, \theta) > 0$, there exists $R > 1$ such that
		$$\exp ((1 - \epsilon)\delta(P, \theta)r^n) \leq |A(re^{\iota\theta})|\leq \exp ((1 + \epsilon)\delta(P, \theta)r^n)$$
		for $r > R$,
		\item for $\theta \in [0,2\pi)\setminus E$ with $\delta(P, \theta) < 0$, there exists $R > 1$ such that
		$$\exp ((1 + \epsilon)\delta(P, \theta)r^n) \leq |A(re^{\iota\theta})| \leq \exp ((1 - \epsilon)\delta(P, \theta)r^n)$$ 
		for $r > R$.
	\end{enumerate}
\end{lemma}
The following lemma gives an estimate of an entire function having order of growth strictly less than half and it is given by Besicovitch \cite{besicovitch}. It is used to prove conclusion of Theorem \ref{mainthm4} and \ref{mainthm3}.

\begin{lemma}\cite{besicovitch}\label{lbesicovitch} Let $f$ be an entire function of finit order $\rho$ where $0<\rho<1/2,$ and $\epsilon>0$ be a given real constant. Then there exist a set $S\subset [0,\infty)$ that has upper density at least $1-2\rho$ such that $|f(z)|>\exp(|z|^{\rho-\epsilon})$ for all $z$ satisfying $|z|\in S.$ 
\end{lemma}
Recently, Pant and Saini \cite{pantsaini} prove following result for an entire function.  It has been used in the proof of Theorem \ref{mainthm4} and \ref{mainthm1}.

\begin{lemma}\cite{pantsaini}\label{lpantsaini}
	Suppose $f$ is a transcendental entire funtion then, there
	exists a set $F\subset(0,\infty)$ with finite logarithmic measure such that for all $z$ satisfying $|z| = r\in F$ and $|f(z)| = M(r, f)$ we have
	$$
		\left|\frac{f(z)}{f^{(m)}(z)}\right| \leq 2r^m,
	$$
	for all $m\in N.$
\end{lemma}

The following Lemma has been used in our examples.
\begin{lemma}\cite{chang}\label{lchang}
	 If $f$ is an entire function of positive lower order of growth $\mu(f)>0$, then there exist a curve $\Gamma$ that goes from a finite point to $\infty$ for which
	 $$\min\left(\frac{1}{2}, \mu(f)\right)\leq\lim\limits_{z\to\infty}\inf\limits_{z\in\Gamma}\frac{\log\log |f(z)|}{\log |z|}.$$
\end{lemma}
The following lemma gives the property of an entire function with Fabry gap and can be found in \cite{long} and \cite{wuzheng}. It is used to prove the result of example \ref{further}.
\begin{lemma}\label{lfabry}
	Let $g(z)=\sum\limits_{n=0}^\infty a_{\lambda_n}z^{\lambda_n}$ be an entire function of finite order with Fabry gap, and $f(z)$ be an entire function with $\rho(f)\in (0,\infty)$. Then for any given $\epsilon \in (0, \rho(f))$, there exists a set $H \subset (1,+\infty)$ satisfying $\overline{logdense}(H) \geq \xi$, where $\xi \in (0, 1)$ is a constant such that for all $|z| = r \in H$, one has
	$$\log M(r, h) > r^{\rho(f)-\epsilon}, \log m(r, g) > (1 - \xi)\log M(r, g),$$ where $M(r, h) = max\{ |h(z)| : |z| = r\}$ , $m(r, g) = min\{ |g(z)| :|z| = r\}$ and $M(r, g) = max\{ |g(z)| : |z| = r\}.$
\end{lemma}

\section{Results}
Gundersen, et. al. \cite{gundersen1992finite}, consider 
$$\max\{\rho(a_1), \rho(a_2), \ldots, \rho(a_{n-1}), \rho(H)\}<\rho(a_0)<\frac{1}{2},$$
and prove that all non-trivial solutions of equation \eqref{nonhomonorder} are of infinite order. Hellerstein et. al. \cite{hellermilesrossi}, improved the condition 
$$\max.\{\rho(a_1), \rho(a_2), \ldots, \rho(a_{n-1}, \rho(H)\}<\mu(a_0)\leq\frac{1}{2}$$
and proved the same result. Recently, Kumar and Saini \cite{kumarsaini}, consider the case 
$$\max\{\rho(a_i), i\neq j, \rho(H) \}<\mu(a_j)<\frac{1}{2}$$ and gives the same conclusion.\\

    \textbf{Theorem A.}\  \cite{kumarsaini}\ \  \emph{Suppose coefficient in equation \eqref{nonhomonorder} satisfy $\max\{\rho(a_0), \rho(a_1), \ldots, \rho(a_{j-1}),$ \linebreak $\rho(a_{j+1}), \ldots , \rho(a_{n-1}), \rho(H(z)) \}<\mu(a_j)<\frac{1}{2}, j=1,2, \ldots, n-1.$ Then, all non-trivial solution of equation \eqref{nonhomonorder} are of infinite order.}\\

  We have replaced the condition of above theorem to 
$$\max\{\rho(a_i), i\neq j, \rho(H) \}<\rho(a_j)<\frac{1}{2}.$$
 The statement of our first theorem is as follows.  
    \begin{theorem}\label{mainthm4}
	Suppose coefficients in equation \eqref{nonhomonorder} satisfy $\max\{\rho(a_0), \rho(a_1), \ldots, \rho(a_{j-1}),$ \linebreak $\rho(a_{j+1}), \ldots , \rho(a_{n-1}), \rho(H(z)) \}<\rho(a_j)<\frac{1}{2}.$ Then, all non-trivial solution of equation \eqref{nonhomonorder} are of infinite order.
\end{theorem}
\begin{proof}
	Suppose $f$ be any non-trivial solution of equation \eqref{nonhomonorder} of finite order. Then, from Lemma \ref{lgundersen} there exist a set $F\subset [0,\infty]$ of finite linear measure such that for all $|z|\not\in F,$ 
	\begin{equation}\label{ffiniteorder(i)}
		\left|\frac{f^{(g)}(z)}{f^{(h)}(z)}\right| \leq |z|^{g(\rho-1+\epsilon)},\ h<g.
	\end{equation}
	Since $f$ is a transcendental entire function, there exist  a set $G\subset(0,\infty)$ with finite logarithmic measure such that for all $z$ satisfying $|z| = r\in G$ and $|f(z)| = M(r, f)$ we have
\begin{equation}\label{eqpantsaini}
		\left|\frac{f(z)}{f^{(m)}(z)}\right| \leq 2r^m,
	\end{equation}	
	for all $m\in N.$ Let $\omega$ and $\delta$ be fixed constant such that $\max\{\rho(a_0), \ldots, \rho(a_{j-1}), \rho(a_{j+1}), \ldots,$ \linebreak $\rho(a_{(n-1)}, H(z))\}<\omega<\delta <\rho(a_j)$. Thus, there exist a constant $R>0$ such that
	\begin{equation}\label{omega}
		|a_i(z)|\leq \exp\{|z|^\omega\} , i=0, 1, \ldots j-1, j+1, \ldots, n-1
	\end{equation}
	and 
	\begin{equation}\label{H}
		|H(z)|\leq \exp\{|z|^\omega\}.
	\end{equation}
	Using Lemma \ref{lbesicovitch}, there exist a set $S\subset[0,\infty)$ that has upper density atleast $1-2\rho(a_j)$ such that $|z|\in S$ satisfying
	\begin{equation}\label{eq<1/2}
		|a_j(z)|>\exp(|z|^\delta).
	\end{equation}
	Using \eqref{nonhomonorder}, \eqref{ffiniteorder(i)}, \eqref{eqpantsaini}, \eqref{omega}, \eqref{H} and \eqref{eq<1/2}, for $|z|\in (G\cup S)\setminus F$ and $|f(z)|=M(r,f)$ we get
	\begin{align*}
		1 & = -\frac{1}{a_j(z)}\frac{f^{(n)}}{f^{(j)}}-\sum_{k=2}^{i-1}\frac{a_k(z)}{a_j(z)}\frac{f^{(k)}}{f^{(j)}}-\sum_{k=i+1}^{n-1}\frac{a_k(z)}{a_j(z)}\frac{f^{(k)}}{f^{(j)}}+\frac{H(z_r)}{a_j(z_r)}\frac{1}{f^{(j)}}  \\
		& \leq \sum_{k=1}^{j-1}\left|\frac{a_k(z)}{a_j(z)}\right|\left|\frac{f^{(k)}}{f^{(j)}}\right|+\sum_{k=j+1}^{n}\left|\frac{a_k(z)}{a_j(z)}\right|\left|\frac{f^{(k)}}{f^{(j)}}\right|+\left|\frac{H(z)}{a_j(z)}\right|\frac{1}{|f^{(j)}|} \\
		& \leq  \left|\frac{f}{f^{(j)}}\right|\sum_{k=1}^{j-1}\left|\frac{a_k(z)}{a_j(z)}\right|\left|\frac{f^{(k)}}{f}\right|+\sum_{k=j+1}^{n}\left|\frac{a_k(z)}{a_j(z)}\right|\left|\frac{f^{(k)}}{f^{(j)}}\right|+\left|\frac{H(z)}{a_j(z)}\right|\frac{1}{|f^{(j)}|} \\
		& \leq  2r^m\sum_{k=1}^{j-1}\frac{\exp\{|z|^\omega\}}{\exp(|z|^\delta)}|z|^{j(\rho-1+\epsilon)}+\sum_{k=j+1}^{n}\frac{\exp\{|z|^\omega\}}{\exp(|z|^\delta)}|z|^{j(\rho-1+\epsilon)}+\frac{\exp\{|z|^\omega\}}{\exp(|z|^\delta)}|\frac{1}{|f^{(j)}|} \\
		& \leq 2r^m n\exp\{|z|^\omega-|z|^\delta)\}|z|^{j(\rho-1+\epsilon)}+\exp\{|z|^\omega-|z|^\delta\}\frac{1}{|f^{(j)}|}\\	
		& \leq \exp\{|z|^\omega-|z|^\delta\}\left(2n|z|^{m+n(\rho-1+\epsilon)}+\frac{1}{|f^{(j)}|}\right)
	\end{align*}
	Since $\exp\{|z|^\omega-|z|^\delta\}\to 0$, we get
	$$1 \leq \exp\{|z|^\omega-|z|^\delta\}\left(2n|z|^{m+n(\rho-1+\epsilon)}+\frac{1}{|f^{(j)}|}\right)\to 0$$
	which is a contradiction.
	
\end{proof}
 The example \ref{egnecemainthm4}, shows that if we skip the condition of Theorem \ref{mainthm4}, then we can get a non-trivial solution of finite order.
\begin{example}\label{egnecemainthm4}
Consider second order linear differential equation
\begin{equation}\label{2ordernonhomo}
f''+A(z)f'+B(z)f=H(z),
\end{equation}
where $A(z)$, $B(z)$ and $H(z)$ are entire functions. Let $f$ be a non- trivial solution of equation \eqref{2ordernonhomo} satisfying the condition  
\begin{equation}\label{egcond}
\rho(A)<\rho(B)<\rho(f)<\frac{1}{2}.
\end{equation}
 Let $\alpha$, $\beta$ and $\gamma$ are real constants such that
$$\rho(A)<\alpha<\beta<\rho(B)<\gamma<\rho(f)<\frac{1}{2}.$$
By definition of order of growth, we have
\begin{equation}\label{egA}
	|A(z)|<e^{|z|^\alpha}.
\end{equation}
Let $\mu(B)=\rho(B)$, then by Lemma \ref{lchang} for $z\in\Gamma$, we have
\begin{equation}\label{egB}
|B(z)|>e^{|z|^\beta}.
\end{equation} 
Since $\rho(f)<\frac{1}{2}$ using Lemma \ref{lbesicovitch} for $S\subset [0, \infty)$ having positive upper density, we have
\begin{equation}\label{egf}
|f(z)|> e^{|z|^\gamma}.
 \end{equation}
 Using equation \eqref{ffiniteorder(i)}, \eqref{egA} and \eqref{egB}, for $z\in\Gamma$, $|z|\in S\setminus F$ and $z\to\infty$, we get
 \begin{eqnarray}\label{o(1)+1}
 \left|\frac{1}{B(z)}\frac{f''}{f}+\frac{A(z)}{B(z)}\frac{f'}{f}+1 \right| & \leq& \frac{1}{|B(z)|}\left|\frac{f''}{f}\right|+\left|\frac{A(z)}{B(z)}\right|\left|\frac{f'}{f}\right|+1 \nonumber \\  
 & \leq& e^{-|z|^\beta}|z|^c+e^{|z|^\alpha-|z|^\beta}|z|^c+1 \nonumber\\  
 & \leq&  o(1)+1.
 \end{eqnarray}
From equation \eqref{2ordernonhomo}, \eqref{egf} and \eqref{o(1)+1},
\begin{align*}
 |H(z)| & = |f(z)||B(z)|\left|\frac{1}{B(z)}\frac{f''}{f}+\frac{A(z)}{B(z)}\frac{f'}{f}+1 \right| \\
 & > e^{|z|^\gamma+|z|^\beta}(o(1)+1) \\
 & > e^{|z|^\gamma(1+o(1))}(o(1)+1) \\
\end{align*}
Thus, $\rho(H)\geq \gamma$, since $\gamma$ is arbitrarily near to $\rho(f)$, $\rho(H)\geq \rho(f).$ But from equation \eqref{2ordernonhomo} and \eqref{egcond}, $\rho(H)\leq \rho(f).$ Thus, $\rho(H) = \rho(f).$ Hence, we have $\rho(A)<\rho(B)<\rho(H)<\frac{1}{2}$ and $f$ is a finite order solution of equation $\eqref{2ordernonhomo}.$ It shows that condition of Theorem \ref{mainthm4} are necessary.   
\end{example}
In Theorem \ref{mainthm4}, there is a limitations on $\rho(H)$ that it should be less than half and also less than $\rho(a_j), i\neq j.$ So the question arise, is it necessary to always have all non-trivial solution of equation \eqref{nonhomonorder} of infinite order, when $$\max\{\rho(a_0), \rho(a_1), \ldots, \rho(a_{j-1}), \rho(a_{j+1}), \ldots , \rho(a_{n-1}) \}<\rho(a_j)<\frac{1}{2}\leq \rho(H)?$$ We have constructed an example to answer this question.
\begin{example}\label{egnecemainthm5}
	Let $f$ be a non-trivial solution of equation \eqref{2ordernonhomo} such that 
	\begin{equation}\label{egmainthm4cond}
	\rho(A)<\rho(B)<\frac{1}{2}\leq \rho(f),
	\end{equation}
satisfying $\mu(f)=\rho(f).$ Let $\alpha$, $\beta$ and $\gamma$ are real constants such that $$\rho(A)<\alpha<\beta<\rho(B)<\frac{1}{2}<\gamma\leq\rho(f).$$
Then, by Lemma \ref{lchang}, $z\in\Gamma$, we have
\begin{equation}\label{egfmainthm4}
	|f(z)|>e^{|z|^\gamma}.
\end{equation} 
Since $\rho(B)<\frac{1}{2}$ using Lemma \ref{lbesicovitch} for $S\subset [0, \infty)$ having positive upper density, we have
\begin{equation}\label{egBmainthm4}
	|B(z)|> e^{|z|^\beta}.
	\end{equation}
Using the same logic as in example \ref{egnecemainthm4} for $z\in \Gamma$, $z\to\infty$ and $|z|\in S\setminus F$, we get
\begin{align*}
	|H(z)| > e^{|z|^\gamma(1+o(1))}(o(1)+1) 
\end{align*}
Thus, $\rho(H)\geq \gamma$, since $\gamma$ is arbitrarily close to $\rho(f)$, $\rho(H)\geq \rho(f).$ But from equation \eqref{2ordernonhomo} and \eqref{egmainthm4cond}, $\rho(H)\leq \rho(f).$ Thus, $\rho(H) = \rho(f).$ Hence, we have $\rho(A)<\rho(B)<\frac{1}{2}\leq \rho(H)$ and $f$ is a finite order solution of equation $\eqref{2ordernonhomo}.$
\end{example}
Thus, example \ref{egnecemainthm5}, shows that we can get finite order solution when $\rho(H)\geq\frac{1}{2}.$ In Theorem \ref{mainthm5}, we have partially overcome this situation by considering the condition that exponent of convergence is strictly less than order of growth and order should be finite. By Hadamard factorization theorem, it is already known that if $\lambda(H)<\rho(H)=n,$ then $H(z)=h(z)e^{P(z)}$, where $h(z)$ is an entire function and $P(z)$ is a polynomial of degree $n$, where $\rho(h)<n.$ 
\begin{theorem}\label{mainthm5}
	Let the coefficient of equation \eqref{nonhomonorder} satisfy $\max\{\rho(a_1), \ldots, \rho(a_{(n-1)}\}<\rho(a_0)<\frac{1}{2}$. and $\lambda(H(z))<\rho(H(z)).$ Then, all non-trivial solution of equation \eqref{nonhomonorder} are of infinite order.
\end{theorem}
\begin{proof}
	Suppose $f$ be any non-trivial solution of finite order of equation \eqref{nonhomonorder}. Then, using Lemma \ref{lgundersen}, we get
	there exists a set $F\subset [0,\infty]$ of finite linear measure such that for all $|z|\not\in F,$  we get \eqref{ffiniteorder(i)}. Let $\omega$ and $\delta$ be fixed constant such that $\max\{\rho(a_1), \ldots, \rho(a_{n-1})\}<\rho(a_0)<\frac{1}{2}$. Thus, there exist a constant $R>0$ such that
	\begin{equation}\label{omegaa_i}
		|a_i(z)|\leq e^{|z|^\omega}, i=1, \ldots, n-1.
	\end{equation}
	 Since $H(z)$ is an finite order entire function satisfying $\lambda(H)<\rho(H),$ then it can be rewritten as $H(z)=h(z)e^{P(z)}$, where $h(z)$ is an entire function, $P(z)$ is a polynomial of degree $m$ and $\rho(h)<\deg P$. Thus, using Lemma \ref{lcritical}, there exist a set $E\subset [0, 2\pi )$ of zero linear measure such that $\delta(P,\theta)<0$,
	\begin{equation}\label{H(z)critical}
		|H(z)|\leq e^{(1+\epsilon)\delta(P,\theta) r^m}.
	\end{equation}
	Using Lemma \ref{lbesicovitch}, there exist a set $S\subset[0,\infty)$ that has upper density atleast $1-2\rho(a_0)$ satisfying for $|z|\in S$
	\begin{equation}\label{eq<1/2a_0}
		|a_0(z)|>e^{|z|^\delta}.
	\end{equation} 
	From equation \eqref{nonhomonorder}, \eqref{ffiniteorder(i)}, \eqref{omegaa_i}, \eqref{H(z)critical} and \eqref{eq<1/2a_0}, for $|z|\in S\cap F$ and $\theta\in[0,2\pi)\setminus E$ along with $\delta(P, \theta)<0$, we can get the following. 
	\begin{align*}
		1 & = -\frac{1}{a_0(z)}\frac{f^{(n)}}{f}-\sum_{k=1}^{n-1}\frac{a_k(z)}{a_0(z)}\frac{f^{(k)}}{f}+\frac{H(z)}{a_0(z)}\frac{1}{f}  \\
		& \leq \frac{1}{|a_0(z)|}\left|\frac{f^{(n)}}{f}\right|+\sum_{k=1}^{n-1}\left|\frac{a_k(z)}{a_0(z)}\right|\left|\frac{f^{(k)}}{f}\right|+\left|\frac{H(z)}{a_0(z)}\right|\frac{1}{|f|} \\
 & \leq  \sum_{k=1}^{n}\frac{e^{|z|^\omega}}{e^{|z|^\delta}}|z|^{j(\rho-1+\epsilon)}+\frac{e^{{(1+\epsilon)\delta(P,\theta) r^m}}}{e^{|z|^\delta}}\frac{1}{|f|} \\
		& \leq  e^{|z|^\omega-|z|^\delta}n|z|^{n(\rho-1+\epsilon)}+e^{(1+\epsilon)\delta(P,\theta) r^m-|z|^\delta}\frac{1}{|f|}\\	
	& \to 0,
	\end{align*}
	which is a contradiction.	
\end{proof}
What would happen when 
\begin{equation}\label{geqhalf}
\max\{\rho(a_0), \rho(a_1), \ldots, \rho(a_{j-1}), \rho(a_{j+1}), \ldots , \rho(a_{n-1}), \rho(H(z)) \}<\rho(a_j) 
\end{equation} and  $\rho(a_j)\geq\frac{1}{2} ?$
 Is all non-trivial solutions of equation \eqref{nonhomonorder} are of infinite order? The next example is answer to this question, where we get a finite order non-trivial solution. In Theorem \ref{mainthm3}, we add some condition on the coefficient and partially solved the problem. 
\begin{example}
	$f(z)=e^z$ is a finite order solution of linear differential equation
	$$f''+e^zf'+e^{2z}f=e^z+e^{-z}-1.$$
	Comparing, it with equation \eqref{2ordernonhomo} $\rho(A)=\rho(H)<\rho(B)$ and it have finite order solution. Thus, it shows that equation \eqref{nonhomonorder} with condition \eqref{geqhalf} can have non-trivial finite order solution.
\end{example}
\begin{theorem}\label{mainthm3}
	Suppose the coefficients in equation \eqref{nonhomonorder} satisfy $\max\{\rho(a_1), \rho(a_2), \ldots, \rho(a_{n-1}), \rho(H)\}<\rho(a_0)$ and $\lambda(a_0)<\rho(a_0).$ Then, all non-trivial solution of equation \eqref{nonhomonorder} are of infinite order.
\end{theorem}
\begin{proof}
	Suppose $f$ be any non-trivial solution of equation \eqref{nonhomonorder} of finite order. Then, from Lemma \ref{lgundersen} there exists a set $F\subset[0, \infty)$ of finite linear measure such that for all $|z|\not\in F$ satisfying \eqref{ffiniteorder(i)}.	Let $\omega$ be any fixed constant such that $\max\{\rho(a_1), \rho(a_2), \ldots, \rho(a_{(n-1)}, H(z))\}<\omega <\rho(a_0)$. Thus, there exist a constant $R>0$ such that satisfying \eqref{omegaa_i} and \eqref{H}. 
	Since $\lambda(a_0)<\rho(a_0)$, thus it can be rewritten as $a_0(z)=h(z)e^{P(z)}$, where $P(z)$ is a polynomial of degree $k$ and $h(z)$ is an entire function satisfying $\rho< m$. Using Lemma \ref{lcritical} 
	\begin{equation}\label{eqcritical}
		\exp.\{(1+\epsilon)\delta(P,\theta)r^k\}\leq |a_o(z)|
	\end{equation}
	Using \eqref{nonhomonorder}, \eqref{ffiniteorder(i)}, \eqref{H},  \eqref{omegaa_i} and \eqref{eqcritical}, we get
	\begin{align*}
		1 & = -\frac{1}{a_0(z)}\frac{f^{(n)}}{f(z)}-\sum_{j=1}^{n-1}\frac{a_j(z)}{a_0(z)}\frac{f^{(j)}}{f(z)}+\frac{H(z)}{a_0(z)}\frac{1}{f(z)}  \\
		& \leq \sum_{j=1}^{n}\left|\frac{a_j(z)}{a_0(z)}\right|\left|\frac{f^{(j)}}{f(z)}\right|+\left|\frac{H(z)}{a_0(z)}\right|\frac{1}{|f(z)|} \\
		& \leq e^{|z|^\omega-\delta(P,\theta)r^m}\left(n|z|^{n(\rho+\epsilon)}+\frac{1}{|f|}\right)\\	
	\end{align*}
	Since $e^{|z|^\omega-\delta(P,\theta)r^m}\to 0$, we get
	$$1 \leq e^{|z|^\omega-\delta(P,\theta)r^m}\left(n|z|^{n(\rho+\epsilon)}+\frac{1}{|f|}\right)\to 0$$
	which is a contradiction.

\end{proof}
We can also get a finite order solution when
 $$\max\{\rho(a_1), \rho(a_2), \ldots, \rho(a_{n-1})\}<\rho(a_0)<\rho(H).$$ 
\begin{example}\label{further}
	Let $f$ be finite order non-trivial solution of equation \eqref{2ordernonhomo}. Let  	
	\begin{equation}\label{egmainthm3cond}
		\rho(A)<\rho(B)<\rho(f),
	\end{equation}
	such that $B(z)$ has Fabry gaps. Let $\alpha$, $\beta$ and $\gamma$ are real constants such that $$\rho(A)<\alpha<\beta<\rho(B)<\gamma<\rho(f).$$ By Lemma \ref{lfabry}, for $|z|\in H$ such that $H\subset (1, \infty)$ of positive upper logarithmic density, we have
	\begin{equation}\label{eqfabry}
		|B(z)|> e^{|z|^\beta}.
	\end{equation}
	Let $f$ be a such solution which satisfy $\mu(f)=\rho(f).$ 
	Then, by Lemma \ref{lchang}, $z\in\Gamma$ and $z\to\infty$, we have
	\begin{equation}\label{egfmainthm4}
		|f(z)|>e^{|z|^\gamma}.
	\end{equation} 
Using the same reason as in example \ref{egnecemainthm4}, for $z\in \Gamma$, $z\to\infty$ and $|z|\in H\setminus F,$ we can get
\begin{align*}
	|H(z)| > e^{|z|^\gamma(1+o(1))}(o(1)+1) 
\end{align*}
Thus, $\rho(H)\geq \gamma$, since $\gamma$ is arbitrarily nearer to $\rho(f)$, $\rho(H)\geq \rho(f).$ But from equation \eqref{2ordernonhomo} and \eqref{egmainthm3cond}, $\rho(H)\leq \rho(f).$ Thus, $\rho(H) = \rho(f).$ Hence, we have $\rho(A)<\rho(B)<\rho(f)$ and $f$ is a finite order solution of equation $\eqref{2ordernonhomo}.$
 
\end{example}

Recently, Pramanik et. al. \cite{pramanik2021growth} consider equation
\begin{equation}\label{eqpramanik}
	a_nf^{(n)}+a_{n-1}(z)f'+\ldots +a_1f'+a_0f=b(z)f+c(z),
\end{equation}
they extend the results of Bela\"{i}di \cite{bela2002idi} to non-homogenous linear differential equations. They consider $0\leq\alpha<\beta$ and $\xi>0$;  $|a_i(z)|\leq e^{\alpha |z|^\xi}, i=0, 1, \ldots, n$ with $a_n\not\equiv 0$, $|c(z)|\leq e^{\alpha |z|^\xi}$ and $b(z)\geq e^{\beta |z|^\xi}$.\\

\textbf{Theorem B.}\ \cite{pramanik2021growth}\ \emph{Let $E$ be a set of complex numbers satisfying $\overline{dens}\{|z|:z\in E\}$ and let $b (z), a_j (z) (j = 0, 1, \ldots, k)$ and $c(z)$ be entire functions such that for some constants $0\leq\alpha<\beta$ and $\xi>0$ we have $b(z)\geq e^{\beta |z|^\xi}$ and $|a_i(z)|\leq e^{\alpha |z|^\xi}, i=0, 1, \ldots, n;$ $|c(z)|\leq e^{\alpha |z|^\xi}$ as $z\to\infty$ for $z\in E.$ Then, every solution $f\not\equiv 0$ of equation \eqref{eqpramanik} is of infinite order.} \\

We have proved this result with standard non-homogenous equation \eqref{nonhomonorder}.
\begin{theorem}\label{mainthm1}
Let $a_i's, 1\leq i \leq n-1$ and $H(z)$ be an entire functions for $0\leq\alpha<\beta$ and $\xi>0$ satisfying $|a_i|\leq e^{\alpha |z|^\xi}$, $|a_j|\geq e^{\beta |z|^\xi}$, for $i\neq j$ and $|H(z)|\leq e^{\alpha |z|^\xi}$ where $z\to\infty$ and $z\in E$, where $E$ is a set of positive upper density. Then, all non-trivial solutions of equation \eqref{nonhomonorder} are of infinite order. 
\end{theorem}
\begin{proof}
	We are going to prove this theorem by contradiction. Let $f$ be any non-trivial solution of equation \eqref{nonhomonorder} of finite order. Using Lemma \ref{lgundersen}, we have \eqref{ffiniteorder(i)} for $z\not\in F\cup[0,1]$, where $F$ is a set of finite logarithmic measure. Using equation \eqref{nonhomonorder}, \eqref{ffiniteorder(i)} and the fact that $|a_i|\leq e^{\alpha |z|^\xi}$, $|a_j|\geq e^{\beta |z|^\xi}, \ i\neq j,$ $i = 0, \ldots, n-1$  and $|H(z)|\leq e^{\alpha |z|^\xi}$ for $0\leq\alpha<\beta$ and $\xi>0$ where $z\to\infty$ and $z\in E$, where $\overline{dens}\{ |z| : z \in E\}>0$, we get the following
	\begin{align*}
		1 & = -\frac{1}{a_j(z)}\frac{f^{(n)}}{f^{(j)}}-\sum_{k=1}^{j-1}\frac{a_k(z)}{a_j(z)}\frac{f^{(k)}}{f^{(j)}}+\frac{H(z)}{a_j(z)}\frac{1}{f^{(j)}}  \\
		& \leq \sum_{k=1}^{j-1}\left|\frac{a_k(z)}{a_j(z)}\right|\left|\frac{f^{(k)}}{f^{(j)}}\right|+\sum_{k=j+1}^{n}\left|\frac{a_k(z)}{a_j(z)}\right|\left|\frac{f^{(k)}}{f^{(j)}}\right|+\left|\frac{H(z)}{a_j(z)}\right|\frac{1}{|f^{(j)}|} \\
		& \leq  \left|\frac{f}{f^{(j)}}\right|\sum_{k=1}^{j-1}\left|\frac{a_k(z)}{a_j(z)}\right|\left|\frac{f^{(k)}}{f}\right|+\sum_{k=j+1}^{n}\left|\frac{a_k(z)}{a_j(z)}\right|\left|\frac{f^{(k)}}{f^{(j)}}\right|+\left|\frac{H(z)}{a_j(z)}\right|\frac{1}{|f^{(j)}|} \\
		& \leq  2r^m\sum_{k=1}^{j-1}\frac{e^{\alpha |z|^\xi}}{e^{\beta |z|^\xi}}|z|^{j(\rho-1+\epsilon)}+\sum_{k=j+1}^{n}\frac{e^{\alpha |z|^\xi}}{e^{\beta |z|^\xi}}|z|^{j(\rho-1+\epsilon)}+\frac{e^{\alpha |z|^\xi}}{e^{\beta |z|^\xi}}|\frac{1}{|f^{(j)}|} \\
		& \leq 2r^m ne^{2(\alpha-\beta)|z|^\xi}|z|^{j(\rho-1+\epsilon)}+e^{(\alpha-\beta)}\frac{1}{|f^{(j)}|}\\	
		& \leq e^{(\alpha-\beta)|z|^\xi}\left(2n|z|^{m+n(\rho-1+\epsilon)}+\frac{1}{|f^{(j)}|}\right)
	\end{align*}
	Since $e^{(\alpha-\beta)|z|^\xi}\to 0$, we get
	$$1 \leq e^{(\alpha-\beta)|z|^\xi}\left(2n|z|^{m+n(\rho-1+\epsilon)}+\frac{1}{|f^{(j)}|}\right)\to 0$$
	which is a contradiction.
\end{proof}

\end{document}